\documentclass{elsarticle}
\usepackage{amsmath}
\usepackage{amsfonts}
\usepackage{amssymb}
\usepackage{amsthm,amsmath}
\usepackage{amssymb,latexsym}
\usepackage{amscd,hyperref}
\usepackage[noend]{algpseudocode}
\usepackage{algorithm}
\usepackage{breqn}
\usepackage{diagbox}
\usepackage{enumitem}
\makeatletter
\renewcommand{\fnum@algorithm}{\fname@algorithm}
\makeatother
\usepackage{color}
\usepackage{graphicx}

\newcommand{\ba}{\begin{array}}
	\newcommand{\ea}{\end{array}}
\newcommand{\bdm}{\begin{displaymath}}
\newcommand{\edm}{\end{displaymath}}
\newcommand{\be}{\begin{equation}}
\newcommand{\ee}{\end{equation}}

\newcommand{\bea}{\begin{eqnarray}}
\newcommand{\eea}{\end{eqnarray}}
\newcommand{\beax}{\begin{eqnarray*}}
	\newcommand{\eeax}{\end{eqnarray*}}

\newtheorem{theorem}{Theorem}[section]

\newtheorem{lemma}[theorem]{Lemma}
\newtheorem{proposition}[theorem]{Proposition}
\title{On a Boundary Updating Method for the Scalar Stefan Problem}
\author{Evangelos F. Magirou} \ead{efm@aueb.gr} 
\author{Paraskevas Vassalos} \ead{pvassal@aueb.gr}
\author{Nikolaos Barakitis} \ead{nickbar@aueb.gr}
\address{Department of Informatics, Athens University of Economics and Business, Patission 76, 10434 Athens, Greece.}

\begin{document}
\begin{abstract}
	We report on a general purpose method for the scalar Stefan problem inspired by the standard boundary updating method used in several existence proofs.  By suitably modifying it we can solve numerically any kind of Stefan problem. We present a theoretical justification of the method and several computational results.
\end{abstract}

\maketitle

\section{Introduction}
In the numerical solution of one dimensional phase change (Stefan) problems involving an infinite rod of ice at zero temperature one usually considers a zero thickness liquid region on which the computation is initiated.  A key issue is how to initiate this computation; a common approach is to determine a starting solution analytically and input the results to a numerical scheme after a small, arbitrary time interval \cite{MiVyn}.  Thus a boundary immobilization scheme \cite{Andreu} can be implemented starting from a definite initial temperature distribution. A criticism to this approach is that a singularity is artificially introduced in the physical equations.  However \cite{MiVyn} shows that by a modification of this method self similar solutions for the entire region can be determined, and demonstrate the validity of their method for several cases well known in the literature.

In this work we report on the application of a boundary updating method inspired by the one used in proving the existence of a solution to the phase change problem as outlined in \cite{Andreu}. One considers a Stefan problem with given, positive heat flux at the initial point. An operator is introduced modifying phase boundaries and such that its fixed points are solutions to the Stefan problem. The operator is essentially the integral of the Stefan condition. One then shows that the operator is continuous in the appropriate norm and hence a fixed point exists. It is also shown that the operator is a contraction for small horizons and hence it potentially provides a method to numerically identify the boundary.  This operator approach is not applicable to the Stefan problem whose boundary condition is specified by a temperature function at the origin since a fundamental property  is not valid in this case and a different proof is given \cite{Friedman}. 

We have reexamined this method with a modification that improves its contraction features, and in fact we show that if the original operator is bounded, an appropriate modification gives a contraction.  This relies on a reversal type property, namely that if a function dominates another, their transforms have an inverted domination property. This shows that starting with a function dominated by the fixed point repeated applications of the operator result in oscillations around the fixed point. Then reasoning as in the scalar case a reduction in the amplitude of the oscillations is possible if instead of the original operator we consider one that is a weighted average of the identity and the operator.  Applying this straightforward idea to standard cases in the literature \cite{Andreu},\cite{MiVyn} we were able to obtain the analytic  without using the extra analysis required therein. Also novel examples were solved giving results consistent with the Stefan condition.  These example were of both boundary types.  Furthermore we were able to solve problems with a Stefan condition corresponding to inhomogeneous ice. We have some results indicating why the method is successful but these rely on reasonable but unproven assumptions.

The next section (Section \ref{FreeBdyPbm}) presents the Stefan problems with particular attention paid to the variations in the boundary conditions and the homogeneity mentioned earlier as well as some relevant results from the literature, especially the operators used in the existence proof in \cite{Andreu}. Section \ref{Iterative} presents the main iterative methods introduced and their justification.  Section \ref{Numerical} includes  the numerical schemes used.  In Section  \ref{AltIterative} we present some iterative methods of a different philosophy based on an optimization approach as in our previous work \cite{AmPut}, show how to implement them and present an argument for their convergence.  Conclusions and further work plans are in the final Section \ref*{Conclusions}.

\section{Statement of the Free Boundary Problem}\label{FreeBdyPbm}
We follow \cite{Andreu} and \cite{MiVyn} in formulating the one dimensional Stefan problem.  The space parameter is $x$ and the time one $t$.  Consider the melting of ice, initially occupying the half line $x\ge0$ as a rod of infinitesimal thickness which is then heated at $x=0$.  The liquid region at time $t$ is from zero to $s(t)$ and we denote it by $Q_s=\{(x,t)|0< x<s(t)\}$.  We are interested in the temperature distribution $U(x,t)$  in the liquid region $Q_s$,  and the solid - liquid interface or boundary $s(t)$. The interface is not known and must be determined so that it satisfies a condition reflecting the latent heat of the melting of ice. In the liquid phase $Q_s$ the temperature $U(x,t)$ satisfies the heat equation
\begin{equation}
\frac{\partial{U}}{\partial{t}}(x,t)=\frac{\partial^2{U}}{\partial{x^2}}(x,t), \qquad 0<x<s(t), \: t>0  \label{eq:Heat}
\end{equation}
while in the solid phase $Q^c_s=\{(x,t)|x\ge s(t)\}$ we assume that the temperature is everywhere at zero. One can consider the case where there exists initially a liquid phase interval at zero temperature, in which case the problem simplifies; we will assume no such interval.   We thus have the initial conditions 
\begin{align}
&s(0)=0, \label{eq:Initial1}\\
&U(0,0)=0\label{eq:Initial2}\\
&U(s(t),t)=0\label{eq:Initial3}
\end{align}

 On the liquid-solid boundary the heat of melting must equal the heat transfer due to the temperature gradient.  Thus if $L=L(x)$ is the latent heat per unit volume, $k=k(x)$ the diffusivity coefficient of the liquid at position $x$ then $-kU_x(s(t),t)\delta t=L\delta s$; setting $\beta=L/k$ we obtain the Stefan condition
\begin{equation}\label{Stefan}
\beta(s(t)) \frac{ds}{dt} = -\frac{\partial{U}}{\partial{x}}(s(t)t). 
\end{equation}

In case $L,k$ depend on the position $x$, a situation that might occur if there are impurities in the solid phase,  there is a space dependence $\beta=\beta(x)$.  We will mostly work with a constant $\beta$ but our numerical methods work for variable $\beta$'s.  

We consider as in the literature two modes of heating, a Dirichlet and a Neumann form and consider two types of condition at $x=0$
\begin{align}
&U(0,t)=g(t),&Dirichlet\; boundary     \label{eq:Dirichlet boundary} \\       
&\frac{\partial{U}}{\partial{x}}(0,t) = h(t)\ge0\qquad &Neumann\; boundary\label{eq:Neumann boundary}
\end{align}
As customary we assume $g$ to be nonnegative and $h$ to be positive.

Let $Q_{s,T}=\{(x,t)|0< x< s(t),0< t< T\}$ and $\overline{Q_{s,T}}$ be the closure of $Q_{s,T}$ .  A solution of the problem consisting of \eqref{eq:Heat},\eqref{eq:Initial1}-\eqref{eq:Initial3},\eqref{Stefan},and either \eqref{eq:Dirichlet boundary} or \eqref{eq:Neumann boundary} in the horizon from zero to $T$ i.e. $Q_{s,T}$ is a pair $(U,s)$ with $s\in C^1((0,T])\cap C([0,T]),\;s(0)=0,\;s\ge0,\;0\le t \le T, \; U\in C(\overline{Q_{s,T}}) \cup C^{2,1}(\overline{Q_{s,T}}),\;U_x \in C(\overline{Q_{s,T}}-\{t=0\})$.  For constant $\beta$ there are standard existence and uniqueness theorems for both the Dirichlet and the Neumann condition as long as the heat flux $h$ is positive and the temperature is nonnegative at the origin \cite{Andreu},\cite{Friedman}, but the proof of the Neumann case is more relevant to our work. 

In the existence proofs the following construction - transformation proves useful \cite{Andreu}:  One considers a function $s$ serving as a candidate for the interface (and hence $s(0)=0$) and then solves a heat equation problem without the Stefan condition and $s$ as the interface. Namely one considers the solution $U^s(x,t)$ of \eqref{eq:Heat},\eqref{eq:Initial1}-\eqref{eq:Initial3} and either \eqref{eq:Dirichlet boundary} or \eqref{eq:Neumann boundary} from zero to $T$ i.e. in the domain $Q_{s,T}$.  We will refer to  the problem of finding a solution for given $s$ and a Dirichlet condition \eqref{eq:Dirichlet boundary} as the \textbf{Fixed Boundary Dirichlet problem}. Finding a solution $U(x,t) \equiv U^s(x.t)$  for $s$ and a Neumann condition  \eqref{eq:Neumann boundary} will be called the \textbf{Fixed Boundary Neumann problem}

One then modifies the boundary $s$ by transforming it through:
\begin{equation}\label{RTransform}
\mathcal{R}(s)(t)=-\int_{0}^{t}U^s_x(s(z),z)dz\qquad0\le t\le T.
\end{equation}
It is easily verified that a fixed point $s^*$ of \eqref{RTransform} is a solution to the Stefan problem.  Now, as in  \cite{Andreu}, using the divergence theorem for the identically zero function $U_{xx}-U_t$ on $Q_{s,T}$ we obtain 
\begin{equation}\label{ModRTransform}
\mathcal{R}(s)(t)=-\int_{0}^{t}U^s_x(0,z)dz-\int_{0}^{s(t)}U^s(x,t)dx \qquad0\le t\le T
\end{equation}
For Neumann type problems the transform becomes
\begin{equation}\label{ModRTransformNeum}
\mathcal{R}(s)(t)=\int_{0}^{t}h(z)dz-\int_{0}^{s(t)}U^s(x,t)dx
\end{equation}
and it can be then shown that $\mathcal{R}$ is continuous in the appropriate norm;  by Schauder's theorem stated in \cite{Andreu} it has a fixed point which is the required solution. Moreover it is also shown in the same reference that  it is a contraction for a small enough horizon $T$ leading thus to a constructive existence proof.  This argument is not applicable for the Dirichlet problem \eqref{eq:Dirichlet boundary} since the expression \eqref{ModRTransform} does not lead to \eqref{ModRTransformNeum} which is crucial for proving continuity of $\mathcal{R}$. A different approach is used to show existence for this case \cite{Friedman}. However we will show that $\mathcal{R}$  after its modification by the divergence theorem \eqref{ModRTransform} has additional properties that  can serve as a basis for a general purpose iterative algorithm, as we show next.

\section{A modified iterative algorithm for free boundary problems}\label{Iterative}

A very useful property of $\mathcal{R}$ is stated as a n exercise in \cite{Andreu}.  We state and prove it next.
\begin{lemma}\label{reversal} \textbf{(Reversal lemma)}
	
	\noindent Consider the Fixed Boundary Neumann Problem as defined in the previous section (Namely the solution $U^s$ of  \eqref{eq:Heat},\eqref{eq:Initial1}-\eqref{eq:Initial3} and the Neumann condition \eqref{eq:Neumann boundary}) and assume that it has a solution for any smooth $s$ starting at the origin. If $s_1,s_2$ are boundary functions such that $s_1(t)\le s_2(t),\;t \in [0,T]$ then  $\mathcal{R}(s_1)(t)\ge \mathcal{R}(s_2)(t)$. 
\end{lemma}
\begin{proof}
	We assumed the existence of $U^s$. If $h>0$ then we will first show that $U^s$ is everywhere nonnegative. Since $U^s$ is zero on the boundary $s$ and it attains its extremes on the parabolic boundary, the minimum must be at some $(x,t)=(0,t^*)$.  If the minimum is negative then $U^s(0,t^*)<0$. But since $U_x^s(0,t^*)=-h(t^*)<0$ the minimum can not occur on the parabolic boundary and thus $U^s$ is nonnegative on $Q_{s,T}$. 
	
	Consider now  the solutions of the above problem in the boundaries  $s_1,s_2$$U^{s_1,\epsilon}$ and $U^{s_2}$  respectively with Neumann conditions  $-U_x^{s_1,\epsilon}(0,t)=h(t)+\epsilon$ ,$-U_x^{s_2}(0,t)=h(t)$.  Also consider their difference $\Delta^\epsilon=U^{s_1,\epsilon}-U^{s_2}$.  On $s^1$, $\Delta^\epsilon\ge0$ and if its minimum is negative it occurs on a zero value of $x$. But this is impossible since $\Delta^\epsilon_x(0,t)=-\epsilon<0$ and thus $\Delta^\epsilon\ge0$.  As $\epsilon\rightarrow 0$   and assuming continuity with $\epsilon$ we obtain $\Delta^0=U^{s_1,0}-U^{s_2}\ge 0$ everywhere. Then using \eqref{ModRTransformNeum} we have
	\begin{align}\label{nnn}
	\mathcal{R}(s_1)(t)-\mathcal{R}(s_2)(t)=&\int_{0}^{s_2(t)}U^{s_2}(x,t)dx-\int_{0}^{s_1(t)}U^{s_1}(x,t)dx= \nonumber \\
	&\int_{0}^{s_1(t)}[U^{s_2}(x,t)-U^{s_1}(x,t)]dx+\int_{s_1(t)}^{s_2(t)}U^{s_2}(x,t)dx.
	\end{align}
	Since both terms are nonnegative we have $\mathcal{R}(s_1)\ge -\mathcal{R}(s_2)$.
\end{proof} 
\begin{center}
	\begin{figure}[ht]
		\includegraphics[width=10cm, height=8cm]{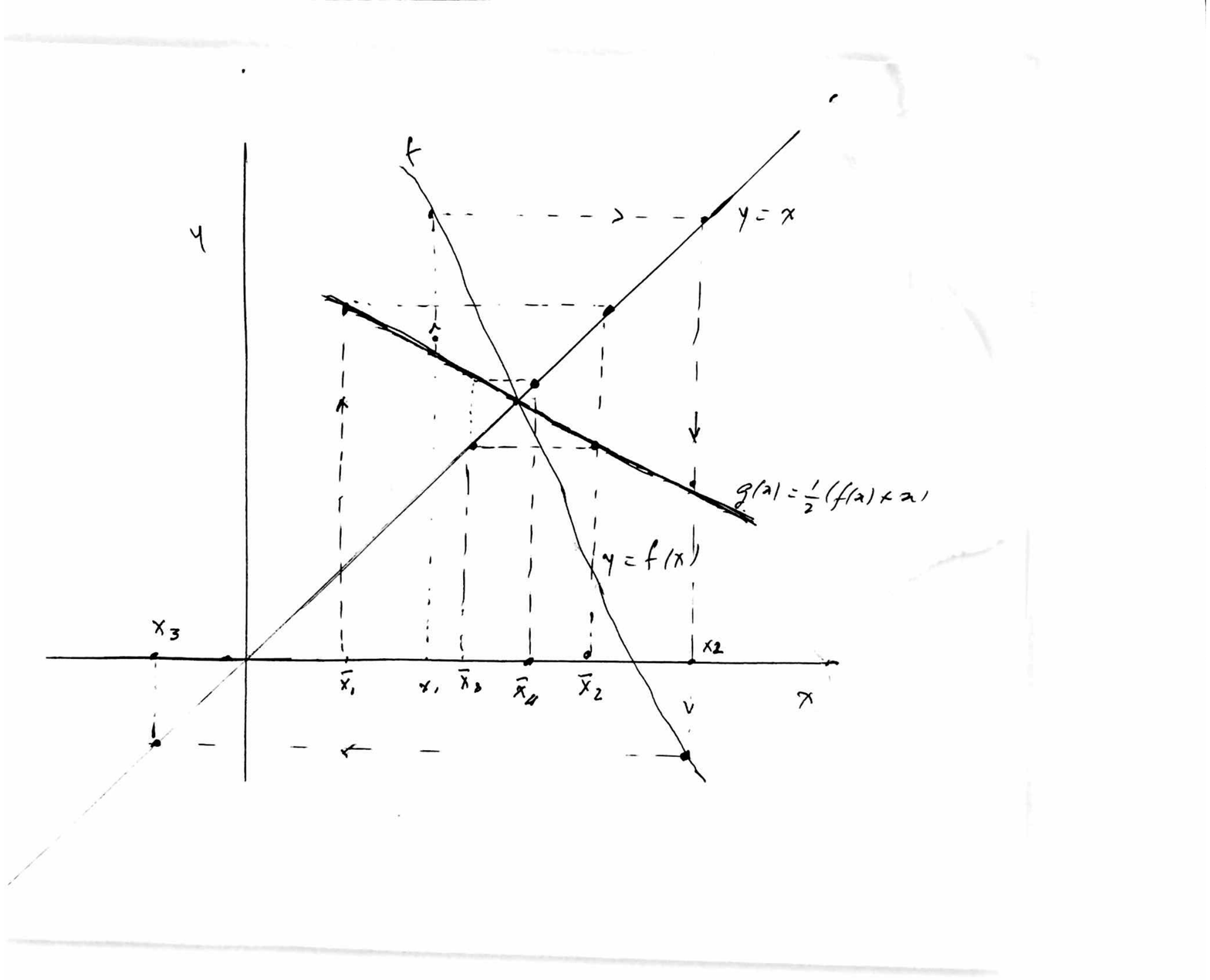}
		\caption{Illustrating the modified boundary transform}
		\label{fig:1}
	\end{figure}
\end{center}
This \textit{reversal lemma} \ref{reversal} is suggestive  of a useful property in one dimensional fixed point algorithms - see Diagram \ref{fig:1} -  where we are interested in computing the fixed point of a scalar function $f$ with a large, negative slope.  The iteration  $x_{n+1}=f(x_n)$ is diverging, but the function  $g(x)=(f(x)+x)/2$ is of smaller absolute value of slope and thus has a better chance of converging to the fixed point through the iteration $\overline{x}_{n+1}=g(\overline{x}_n)$.

We thus introduce the operators $\mathcal{P}^\alpha (s)$ with $\alpha \in [0,1]$ as 
\begin{align}\label{ModOperator}
\mathcal{P}^\alpha (s)(t)=\alpha\mathcal{R}(s)(t)+(1-\alpha)s(t)
\end{align}
\noindent In most of our implementations $\alpha$ will be taken as $1/2$ but different values might be necessary for convergence. Note that $s^*$ is a fixed point of both $\mathcal{P}$ and $\mathcal{R}$. An indicative result for $\alpha=1/2$ is as follows:
\begin{proposition}
	Let $s^*$ be the solution of the Stephan problem with the Neumann condition and let $s\le s^*$ for all $t$. Then for all $t$
	\begin{align}
	s-s^*\le \frac{s-s^*}{2}\le \mathcal{P}^{1/2}(s)-s^*\le \frac{\mathcal{R}(s)-s^*}{2}\le\mathcal{R}(s)-s^*.\end{align}
If $s\ge s^*$  for all $t$
\begin{align}
s-s^*\ge \frac{s-s^*}{2}\le s^*-\mathcal{P}^{1/2}(s)\ge \frac{\mathcal{R}(s)-s^*}{2}\ge\mathcal{R}(s)-s^*.
\end{align}
\end{proposition}
\begin{proof}
	If $s\le s^*$
\begin{align}
\mathcal{P}^{1/2}(s)-s^*= \frac{1}{2}(\mathcal{R}(s)+s)-\frac{1}{2}(\mathcal{R}(s^*)+s^*)=\frac{1}{2}(\mathcal{R}(s)-s^*)+\frac{1}{2}(s-s^*)\le  \frac{1}{2}(\mathcal{R}(s)-s^*) 
\end{align}	
and 
\begin{align}
\mathcal{P}^{1/2}(s)-s^*= \frac{1}{2}(\mathcal{R}(s)+s)-s^*=\frac{1}{2}(\mathcal{R}(s^*)-s^*)+\frac{1}{2}(s-s^*)\ge  \frac{1}{2}(s-s^*) 
\end{align}
A similar proof in case $s>s^*$.
\end{proof}
\noindent  A similar result is valid if $s$ dominates the equilibrium $s^*$. The proof is in a sense a special case of the next proposition modulo a boundedness assumption on the operator $\mathcal{R}$. The proposition shows that our modified operator $\mathcal{P}$ has better  contraction properties than $\mathcal{R}$. We generalize the proposition in our next result.  We use the norm $\|s_1-s_2\|\equiv \min_t|s_1(t)-s_2(t)|$ for $t\le T$.
\begin{proposition}\label{prop:GammaBound}
	Assume $s_1\le s_2$ and that $\mathcal{R}$ satisfies the reversal property of Lemma \ref{reversal}, while $\mathcal{P}^\alpha (s)(t)=\alpha\mathcal{R}(s)(t)+(1-\alpha)s(t)$ as in \eqref{ModOperator}.  Assuming that $\mathcal{R}$ is bounded, i.e. $\|\mathcal{R}(s_1)-\mathcal{R}(s_2)\|\le \gamma\|s_1-s_2\|$ and $\gamma$ is finite, then 
	\begin{equation}
\|\mathcal{P}^{\frac{1}{\gamma+1}}(s_1)-\mathcal{P}^{\frac{1}{\gamma+1}}(s_2)\|\le \frac{\gamma}{\gamma+1}\|s_1-s_2\|\le\|s_1-s_2\|
	\end{equation}
	In terms of the operators $\mathcal{P}^\alpha$ we have $\alpha= \frac{1}{\gamma+1}$.
\end{proposition}

\begin{proof}
Let $s_1\le s_2$ all $t$.  Then for all $t$ 
\begin{equation}
(1-\alpha)(s_1-s_2)\le\mathcal{P}^\alpha(s_1)-\mathcal{P}^\alpha(s_2)=\alpha(\mathcal{R}(s_1)-\mathcal{R}(s_2))+(1-\alpha)(s_1-s_2)\le \alpha(\mathcal{R}(s_1)-\mathcal{R}(s_2)) 
\end{equation}
since $s_1\le s_2,\;\mathcal{R}(s_1)\ge\mathcal{R}(s_2)$.  

Now if $\mathcal{P}^\alpha(s_1)-\mathcal{P}^\alpha(s_2)\ge0$ at some $t$ then $$|\mathcal{P}^\alpha(s_1)(t)-\mathcal{P}^\alpha(s_2)(t)|\le \alpha|\mathcal{R}(s_1)(t)-\mathcal{R}(s_2)(t)|\le \alpha\max_t|\mathcal{R}(s_1)(t)-\mathcal{R}(s_2)(t)|= \alpha\|\mathcal{R}(s_1)-\mathcal{R}(s_2)\|\le\alpha\gamma\|\mathcal{R}(s_1)-\mathcal{R}(s_2)\|$$ and otherwise if $\mathcal{P}^\alpha(s_1)-\mathcal{P}^\alpha(s_2)\ge0$ then  $$|\mathcal{P}^\alpha(s_1)(t)-\mathcal{P}^\alpha(s_2)(t)|\le (1-\alpha) |s_1(t)-s_2(t)|\le (1-\alpha)\max_t |s_1(t)-s_2(t)|=(1-\alpha) \|s_1-s_2\| $$
Hence $$|\mathcal{P}(s_1)(t)-\mathcal{P}(s_2)(t)|\le \max(\alpha\gamma\|s_1-s_2\|,(1-\alpha)\|s_1-s_2\|)$$
Since the above relation is true for all $t$ we have
\begin{equation*}
\max_t|\mathcal{P}(s_1)(t)-\mathcal{P}(s_2)(t)|\equiv\|\mathcal{P}(s_1)-\mathcal{P}(s_2)\|\le  \max(\alpha\gamma,1-\alpha)\|s_1-s_2\|
\end{equation*}
and finally
\begin{equation}
\|\mathcal{P}(s_1)-\mathcal{P}(s_2)\|\le \min_\alpha \max(\alpha\gamma,1-\alpha)\|s_1-s_2\|=\frac{\gamma}{\gamma+1}\|s_1-s_2\|
\end{equation}
The min-max relation in the last equality occurs when the two terms $\alpha\gamma,1-\alpha$ are equal and hence $\alpha=\frac{1}{1+\gamma}$. The above inequality is symmetric in $s_1,s_2$ so it is valid whenever one boundary dominates the other.
\end{proof}
In the existence proof for the Neumann version in \cite{Andreu} it is shown that the $\mathcal{R}$ operator is indeed bounded so the $\mathcal{P}$ operators can be used as contraction operators to locate the fixed point. We next show that even in the Dirichlet problem the $\mathcal{R}$ operator has the reversal property, so if one can show that it is bounded this provides a justification of the use of $\mathcal{P}$ to locate the fixed point.
\begin{lemma}
	Consider the fixed boundary Dirichlet problem with $U^s(0,t)=g(t)$, $g$ nonnegative and nondecreasing.  Then for a given boundary $s$ the solution $U^s$ satisfies $U^s_x(0,t)\le 0$.
\end{lemma}
\begin{proof}
	Consider the problem for $Q_{s,T}$ i.e. $t\in[0,T]$.The maximum of $U^s$ is on the parabolic boundary and since $U^s(s(t),t)=0$ the maximum is at $x=0,\;t=T$ where $U^s(0,T)=g(T)$ given that  $g$ is nondecreasing.  Hence $U^s_x(0,T)\le0$ for otherwise the maximum would not be on the parabolic boundary. Since $g$ is nondecreasing the argument can be repeated for any value of $T$.
\end{proof}
We will not use the above lemma, but it is interesting in its own right.  It is not valid for $g$'s that have both increasing and decreasing regions.
\begin{lemma}\label{lem:DerivDirichlet}
	For the Fixed Boundary Dirichlet problems with boundaries $s_1\le s_2$ (i.e. for which  $U^{s_1}(0,t)=U^{s_2}(0,t)=g(t)$) we have $U^{s_1}_x(0,t)\le U^{s_2}_x(0,t)$.
\end{lemma}
\begin{proof}
	Consider $\Delta U(t)=U^{s_2}(x,t)-U^{s_1}(x,t)$ up to $s_1$, which vanishes for $x=0$ and is nonnegative on $s_1$. The minimum of $\Delta U$ is on the parabolic boundary and in particular for some zero $x$ and some $t$.  But then  $\Delta U_x(0,t)$ must be nonnegative for otherwise there is a minimum in the parabolic interior.  Therefore   $\Delta U_x(0,t)=U^{s_2}_x(0,t)-U^{s_2}_x(0,t)\ge 0$ as required. 
\end{proof}
The main result for Dirichlet problems is the following:
\begin{proposition}
	The operator $\mathcal{R}$ as defined in \eqref{RTransform} for the Fixed Boundary Dirichlet problem with nonnegative boundary condition $U(0,t)=g(t)\ge0$ has the reversal property of Lemma \ref{reversal}, i.e. if $s_1,s_2$ are boundary functions and $s_1(t)\le s_2(t),\;t \in [0,T]$ then  $\mathcal{R}(s_1)(t)\ge \mathcal{R}(s_2)(t)$. 
\end{proposition}
\begin{proof}
 Consider the boundary functions (smooth and starting at the origin) $s_1,s_2$ and such that $s_1(t)\le s_2(t),\;t \in [0,T]$ and the corresponding solutions $U^{s_1}(x,t),U^{s_2}(x,t)$ on for $x\le s_1(t)$.  On $s_1$ we have $U^{s_2}(s_1(r),t)-U^{s_1}(s_1(r),t)\ge 0$ and at zero $U^{s_2}(0,t)-U^{s_1}(0,t)=g(t)-g(t)=0$.  Hence by the minimum principle $U^{s_2}(x,t)-U^{s_1}(x,t)\ge 0$ for $x\le s_1(t)$.  We also have by lemma \ref{lem:DerivDirichlet} $U^{s_1}_x(0,t)\le U^{s_2}_x(0,t)$.  Hence using expression \eqref{ModRTransform} for $\mathcal{R}$ we have
 \begin{equation}
 \mathcal{R}(s_2)(t)-\mathcal{R}(s_1)(t)=-\int_{0}^{t}[U^{s_2}_x(0,\tau)-U^{s_1}_x(0,\tau)]d\tau-\int_{0}^{s_2(t)}U^{s_2}(x,t)dx+\int_{0}^{s_1(t)}U^{s_1}_x(x,t)dx=
 \end{equation}
 The first integral is non negative by the derivatives' property at zero.  The last two integrals are 
  \begin{equation}
-\int_{0}^{s_2(t)}U^{s_2}(x,t)dx+\int_{0}^{s_1(t)}U^{s_1}(x,t)dx=\int_{0}^{s_1(t)}[U^{s_1}(x,t)-U^{s_2}(x,t)]dx-\int_{s_1(t)}^{s_2(t)}U^{s_2}(x,t)dx\ge 0
 \end{equation}
Hence $\mathcal{R}$ has the desired property
\end{proof}
To some extent a similar treatment is possible for the nonhomegeneous problem where the Stefan condition is $-U_x(s(t),t))=\beta(s(t))\dot{s}(t)$.  Then we modify the operator $\mathcal{R}$ as follows (to be completed...)

\subsection{Iterative algorithm based on the Boundary Immobilization Technique}
Set,
\begin{itemize}
	\item[] T=final\:time
	\item[] dx=spatial\:step
	\item[] dt=time\:step
	\item[] h\_function=h(t)
	\item[] t\_vec=0:td:T
	\item[] h\_vec=h\_function(t\_vec)
	\item[] s\_vec 
	\item[] error=1
	\item[] counter = 0
\end{itemize}
such that $1/dx$, $T/dt$ be integer and $s\_vec$ be the vector of boundary of equal length with $t\_vec$, initially arbitrary.
\begin{itemize}
	\item[] \textbf{while} \: error $\geq$ tol \:\& counter  $\leq$ upCounterBound
	\begin{itemize}
		\item[] F=temperatureDistribution(T,dx,dt,s,h)
		\item[] \textbf{for} i=1:length(s)
		\begin{itemize}
			\item[] s1(i) = -trapz(h\_vec(1:i))dt - trapz(F(:,i))s(i)dx
		\end{itemize}
		\item[] \textbf{end for}
		\item[] error = norm(abs(s1-s),inf)
		\item[] counter = counter+1 
		\item[] s=s1
	\end{itemize}
	\item[] \textbf{end while}
\end{itemize}

\subsection{Temperature distribution for given T,dx,dt,s,h}
We will use the Crank-Nicolson scheme to implement the function temperatureDistribution(T,dx,dt,s,h) that returns the temperature distribution for given T,dx,dt,s,h.

\section{Numerical schemes and results}\label{Numerical}
\subsection{Boundary immobilization technique-BIM}
As the boundary varies in time, in order to keep the number of spatial nodes fixed and any finite difference scheme can be applied,  either a changing spatial step, either a coordinate transformation technique must be used.We here use the Boundary Immobilization technique, that fixes the moving boundary by using a fixed co-ordinate system in space. So, we set

\begin{align}
&\xi=\frac{x}{s(t)}, \: T=F(\xi,t),                  \label{eq:spatial transformation}
\end{align}
so that the equations (\ref{eq:Heat})-(\ref{eq:Neumann boundary}) become

\begin{align}
&\frac{\partial^2{F}}{\partial{\xi^2}}=s^2\frac{\partial{F}}{\partial{t}}-s\xi\frac{ds}{dt}\frac{\partial{F}}{\partial{\xi}} \label{eq:standard parabolic tranformed}\\
&s(0)=0,\\
&F(\xi,0)=0, \\
&\beta\frac{ds}{dt}=-\frac{\partial{F}}{\partial{\xi}}|_{\xi=1}  
\end{align}
subject to 
\begin{align}
&F(0,t)=g(t), \:or                                               \label{eq:Dirichlet boundary transformed} \\       
&\frac{\partial{F}}{\partial{\xi}}|_{\xi=0} = s(t)h(t)=\hat{h}(t)                     \label{eq:Neumann boundary transformed}
\end{align}

Let for symplicity $s^2=z$, that is $s\frac{ds}{dt}=\frac{dz}{2dt}$, and
\begin{align}
&N=\frac{1}{\Delta \xi},\:\: \xi_i = i\Delta\xi \:\: i=0...N \\
&M =\frac{\mathbb{T}}{\Delta t}, \:\: t^{n}=n\Delta t \:\: n=0...M
\end{align}
where $\mathbb{T}$ is the final time.

Then, using the Crank-Nicolson scheme involving a central difference at time $t^{n-\frac{1}{2}}$ and a second order central difference for space derivative for the discretization of \ref{eq:standard parabolic tranformed}, we have

\begin{align}
\frac{1}{2}\left( \frac{F_{i+1}^{n}-2F_{i}^{n}+F_{i-1}^{n}}{\Delta \xi ^{2}} \right)+\frac{1}{2}\left( \frac{F_{i+1}^{n-1}-2F_{i}^{n-1}+F_{i-1}^{n-1}}{\Delta \xi ^{2}} \right) =&z^{n-\frac{1}{2}}\left( \frac{F_{i}^{n}-F_{i}^{n-1}}{\Delta t} \right) \nonumber \\
&-\frac{\xi_i}{2}\left(\frac{dz}{dt}\right)^n \left[\frac{1}{2}\left( \frac{F_{i+1}^{n}-F_{i-1}^{n}}{2\Delta \xi} \right) +\frac{1}{2}\left( \frac{F_{i+1}^{n-1}-F_{i-1}^{n-1}}{2\Delta \xi} \right)\right] \Rightarrow \nonumber \\
(F_{i+1}^{n}-2F_{i}^{n}+F_{i-1}^{n})+(F_{i+1}^{n-1}-F_{i}^{n-1}+F_{i-1}^{n-1})=& z^{n-\frac{1}{2}} \frac{2\Delta \xi^2}{\Delta t}(F_{i}^{n}-F_{i}^{n-1}) \nonumber \\
&-\frac{\xi_i}{4}\left(\frac{dz}{dt} \right)^n\Delta \xi[(F_{i+1}^{n}-F_{i-1}^{n})+(F_{i+1}^{n-1}-F_{i-1}^{n-1})] \nonumber
\end{align}
where $z^{n}=z(d\Delta t)$, $z^{n-\frac{1}{2}}=\frac{z^{n}+z^{n-1}}{2}$ and $\left(\frac{dz}{dt} \right)^n = \frac{z^{n}-z^{n-1}}{\Delta t}$.
\newline
Setting $\rho^{n}=z^{n-\frac{1}{2}} \frac{2\Delta \xi^2}{\Delta t}$ and $\sigma_{i}^{n}=\frac{\xi_i}{4}\left(\frac{dz}{dt} \right)^n\Delta \xi$ the above equation becomes,

\begin{align}
& (F_{i+1}^{n}-2F_{i}^{n}+F_{i-1}^{n})+(F_{i+1}^{n-1}-2F_{i}^{n-1}+F_{i-1}^{n-1})= \rho^{n}(F_{i}^{n}-F_{i}^{n-1}) -\sigma_{i}^{n}[(F_{i+1}^{n}-F_{i-1}^{n})+(F_{i+1}^{n-1}-F_{i-1}^{n-1})] \Rightarrow \nonumber \\
& (1-\sigma_{i}^{n})F_{i-1}^{n}-(2+\rho^{n})F_{i}^{n}+(1+\sigma_{i}^{n})F_{i+1}^{n} = -(1-\sigma_{i}^{n}) F_{i-1}^{n-1}+(2-\rho^{n})F_{i}^{n-1}-(1+\sigma_{i}^{n})F_{i+1}^{n-1}  \label{eq:Equation i=1..N-1}
\end{align}

The above equation is valid  for $i=1...N-1$, while for $i=0$, taking into account that $\sigma_{i}^{n}=0$ and replacing the term $F_{-1}^{n}$ using the relationship $\frac{F_{1}^{n}-F_{-1}^{n}}{2\Delta \xi}=h^{n}s^{n}\Rightarrow F_{-1}^{n}=F_{1}^{n}-2 h^{n}s^{n}\Delta \xi$, whith $s^{n}=s(n\Delta t)$, $h^{n}=h(n\Delta t)$.So the equation becomes

\begin{align}
&-(2+\rho^{n})F_{0}^{n}+2F_{1}^{n} = (2-\rho^{n})F_{0}^{n-1}-2F_{1}^{n-1} + 4\Delta \xi \frac{s^{n}h^{n}+s^{n-1}h^{n-1}}{2} \label{eq:Equation i=0}
\end{align}

Since $F(s(t))=0$, the equation for $i=N$ is $F_{N}^{n}=F_{N}^{n-1}=0$ or 
\begin{align}
-(2+\rho^{n})F_{N}^{n}=(2-\rho^{n})F_{N}^{n-1}=0   \label{eq:Equation i=N}
\end{align}

Let now $R^{n}$, $L^{n}$ $\in \: \mathbb{R}^{(N+1) x (N+1)}$, defined as,

\begin{align}
L^{n}=\left[
\begin{array}{ccccccccc}
-(2+\rho^{n})    &2                 \\
1-\sigma_{1}^{n} &-(2+\rho^{n})    &1+\sigma_{1}^{n} \\
&1-\sigma_{2}^{n} &-(2+\rho^{n})&1+\sigma_{2}^{n}  \\
&                 &\ddots           &\ddots           &\ddots \\
&                 &                 &\ddots           &\ddots          &\ddots \\
&                 &                 &1-\sigma_{N-1}^{n} &-(2+\rho^{n})&1+\sigma_{N-1}^{n}\\
&                 &                 &                   &0            &-(2+\rho^{n})\\
\end{array}
\right],
\label{Matrix:Left matrix}
\end{align}

\begin{align}
R^{n}=\left[
\begin{array}{ccccccccc}
(2-\rho^{n})      &-2                &  \\
-1+\sigma_{1}^{n} & (2-\rho^{n})     &-1-\sigma_{1}^{n} \\
&-1+\sigma_{2}^{n} & (2-\rho^{n})    &-1-\sigma_{2}^{n}  \\
&                  &\ddots           &\ddots           &\ddots  \\
&                  &                 &\ddots           &\ddots           &\ddots \\
&                  &                 &-1+\sigma_{N-1}^{n} & (2-\rho^{n})&-1-\sigma_{N-1}^{n}\\
&                  &                 &                    &0            &(2-\rho^{n})\\
\end{array}
\right],
\label{Matrix:Right matrix}
\end{align}

and $E^{n} \: \in \mathbb{R}^{N+1}$ defined as $E^{n}=[4\Delta \xi \frac{s^{n}h^{n}+s^{n-1}h^{n-1}}{2},0, \dots ,0]^{T}$
\newline
Now, given the temperature distribution $f^{n-1}=[F_{0}^{n-1},F_{1}^{n-1}, \dots ,F_{N}^{n-1}]^{T}$, the equations (\ref{eq:Equation i=0})-(\ref{eq:Equation i=N}) can be written in matrix form

\begin{align}
L^{n}f^{n}=R^{n}F^{n-1}+E^{n} \Rightarrow f^{n}=(L^{n})^{-1}(R^{n}F^{n-1}+E^{n})
\end{align}
for $n=1\dots M$ and $f^{0}=[0,0, \dots ,0]^{T}$ 

\subsubsection{Implementation}
\textbf{temperatureDistribution(T,dx,dt,s,h)}
\begin{itemize}
	\item[] N = $\frac{1}{dx}$ + 1
	\item[] M = $\frac{1}{dx}$ + 1
	\item[] F = zeros(N,M)
	\item[] \textbf{for} n=1:M 
	\begin{itemize}
		\item[] Define $R^n$
		\item[] Define $L^n$
		\item[] Define $E^n$
		\item[] $F(:,n) = (R^n)^{-1}(L^nF(:,n-1)+E^n)$
	\end{itemize}
	\item[] \textbf{end for} 
	\item[] \textbf{return F}
\end{itemize}
%%%%%%%%%%%%%%%%%%%%%%%%%%%%%%%%%%% ADD HERE
\subsection{Numerical Results}\label{Results}
We consider three examples taken from \cite{MiVyn} with the following boundary conditions at $x=0$.
\[{}
 (i)T=e^t-1, \quad  (ii)\frac{\partial{T}}{\partial{x}} = -e^t, \quad  (iii)T= 1-\epsilon \sin(\omega t),
\]
where the parameters $\epsilon$ and $\omega$ in $(iii)$ represent the amplitude and the thermal oscillation. The solution for the first two boundary conditions is known and it is
\[
T = e^{t-x}-1, \quad s(t)=t
\] 
for both of them, provided $\beta=1$. The solution for boundary condition $(iii)$ is unknown.
For the two first examples we compare the approximate solution found using the iterative algorithm with exact solution for both the $s(t)$ and temperature distribution at final time, whilst for the third example we compare the derivative of boundary $s$  with the $\frac{\partial F}{\partial x}$ at $x=s(t)$ as evidence that the Stefan condition at the boundary holds for approximate solution. 

Following the \cite{MiVyn} and \cite{Savovi2003} we define
\[
\Delta\xi_k = 2^{-k}\Delta\xi_0, \quad \Delta\xi_0 = 0.1
\]
\[
\xi_{i,k}=i\Delta\xi_k, \quad i= 0,1,\dots \frac{1}{\Delta\xi_k}.
\]
 
For the examples that the analytical solution is known $T$ we define the error at time $n$  as follows
\[
E_k^n=\left( \Delta\xi_k \sum_{i=0}^{\frac{1}{\Delta\xi_k}}(F(\xi_{i,k},t^n)-T(s^n\xi_{i,k},t^n))^2 \right)^{1/2},
\]
$F$ being the numerical solution and $T$ being the exact. The order of accuracy of the solution is defined as the number 
\[
p=\frac{ln(E_k^n/E_{k+1}^n)}{ln(\Delta\xi_k \ \Delta\xi_{k+1})}
\] 
if that number exists.
\begin{table}[!h]
\centering
\begin{tabular}{ccccc}
          & \multicolumn{2}{c}{Boundary condition (i)}   &\multicolumn{2}{c}{Boundary condition (i)}\\
\cline{2-5}
$\Delta\xi$       & $E^n$        & $p$        & $E^n$                & $p$       \\
\hline
 $1/10$   & $2.21\times 10^{-3}$ &            & $7.03\times 10^{-4}$ &           \\
 $1/20$   & $5.35\times 10^{-4}$ &  $2.050$   & $1.72\times 10^{-4}$ &  $2.024$  \\
 $1/40$   & $1.31\times 10^{-4}$ &  $2.029$   & $4.29\times 10^{-5}$ &  $2.010$  \\
 $1/80$   & $3.22\times 10^{-5}$ &  $2.021$   & $1.06\times 10^{-5}$ &  $2.004$  \\
 $1/160$  & $7.84\times 10^{-6}$ &  $2.041$   & $2.66\times 10^{-6}$ &  $2.002$  \\
\hline
\end{tabular}
\caption{The error of the numerical solution of $F$, fixed at time $t^n=1$ and the order of accuracy for boundary condition $(i)$ and $(ii)$. }
\end{table}

\begin{figure}[!h]
\centering
\includegraphics[width=.48\textwidth]{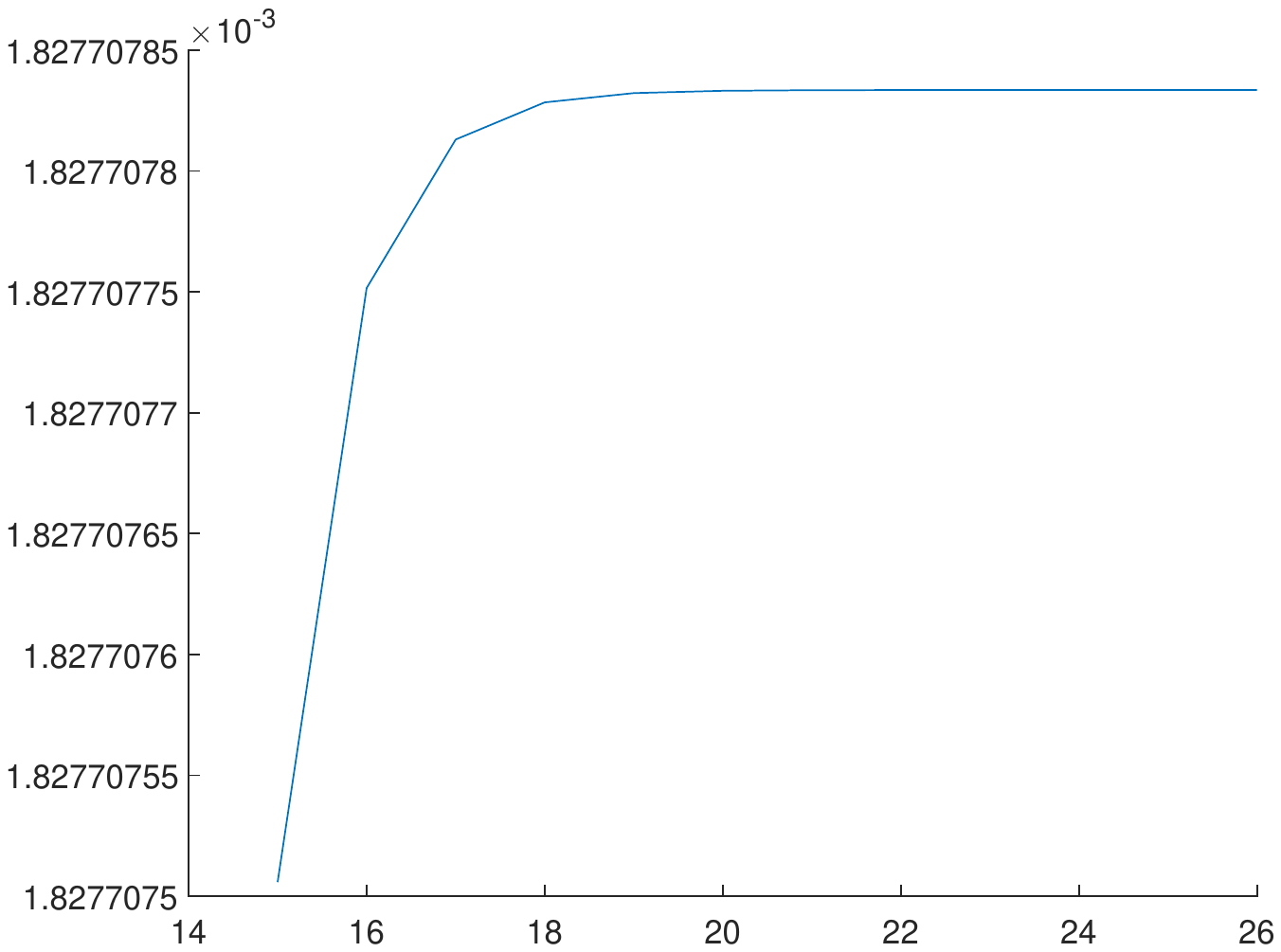}
\includegraphics[width=.48\textwidth]{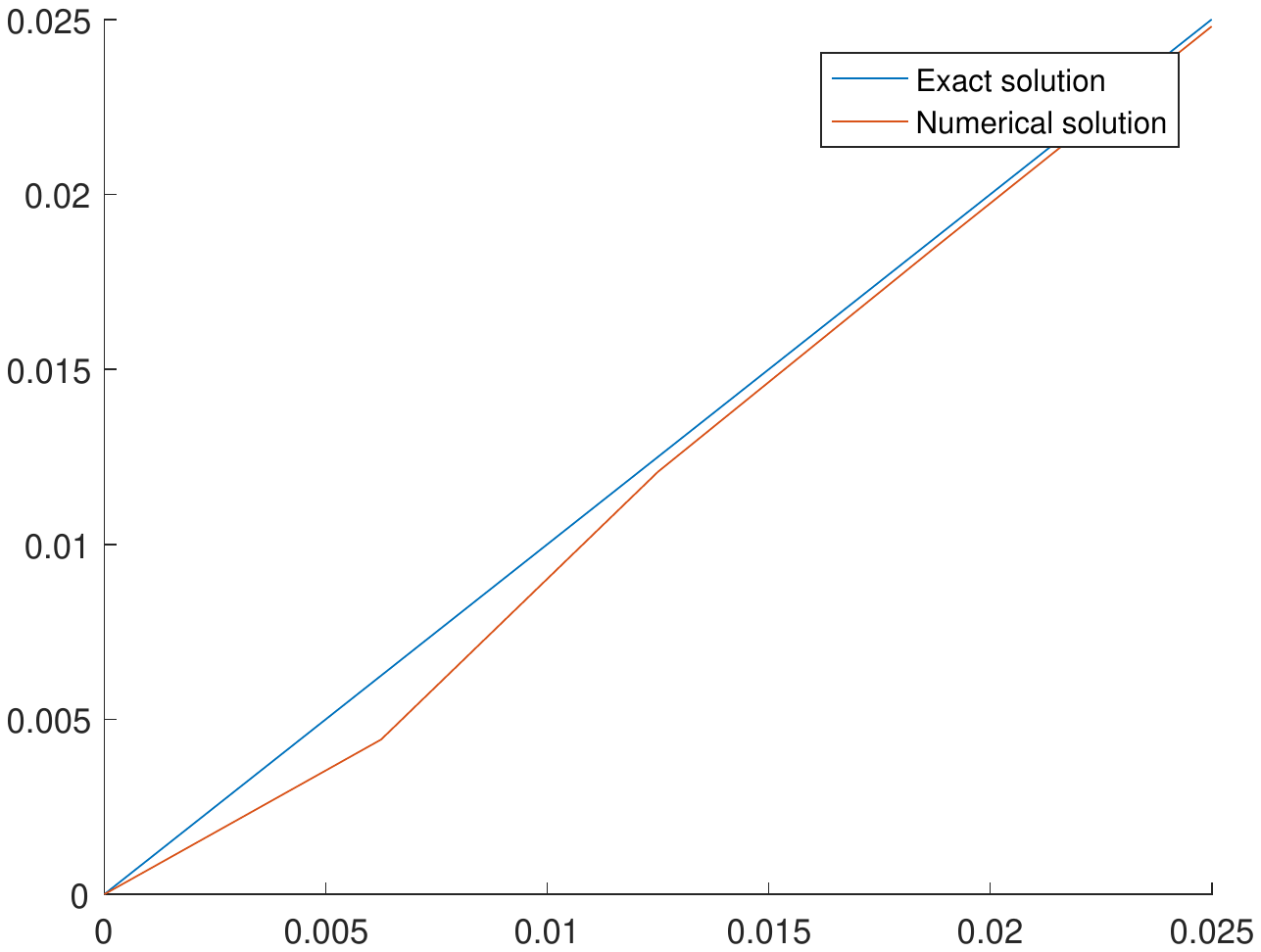}
\caption{The error versus iteration for the boundary found using the iterative method with boundary condition $(i)$. The error is located near the start.}
\label{fig:example1}
\end{figure}

\begin{figure}[!h]
\centering
\includegraphics[width=.48\textwidth]{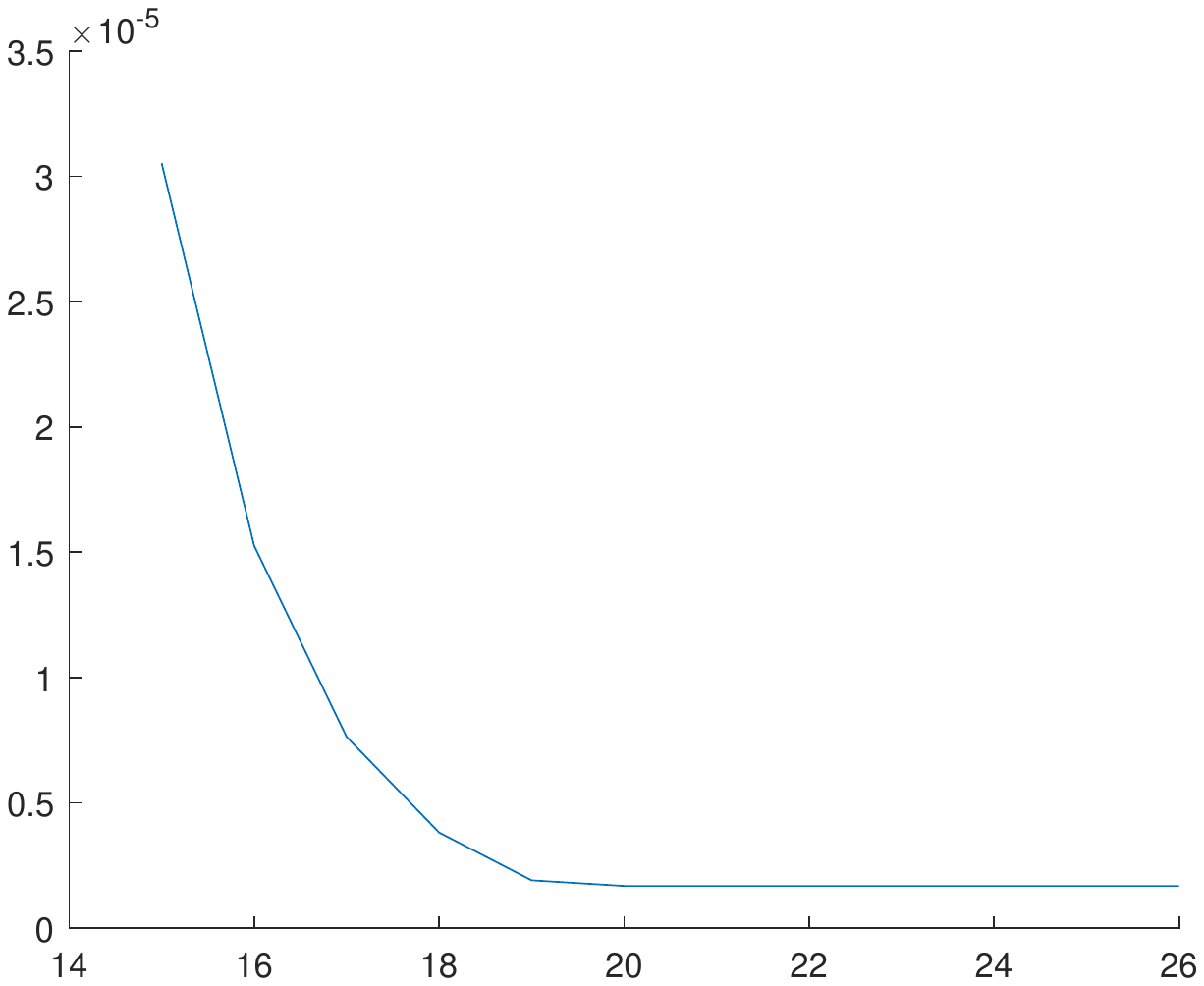}
\includegraphics[width=.48\textwidth]{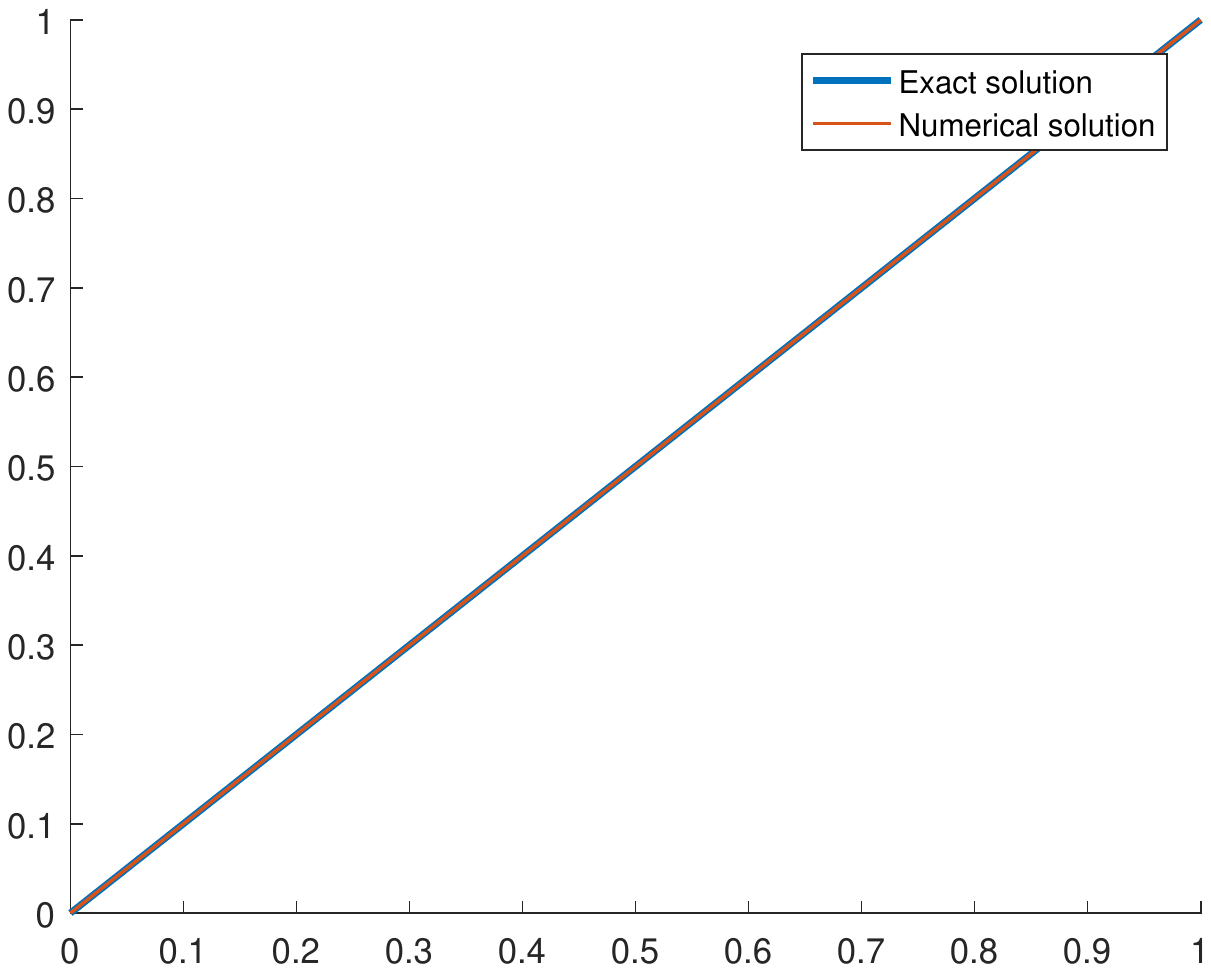}
\caption{The error versus iteration for the boundary found using the iterative method with boundary condition $(ii)$. The approximation is good at all time intervall.}
\label{fig:example2}
\end{figure}

Figure \ref{fig:oscillting boundariesfig:example2} confirmes the oscillating nature of the succesive boundaries. For this example we have removed the average rule and the new boundary have been calculated by straight use of  the rule (\ref{ModRTransformNeum}).

\begin{figure}[!h] 
\centering
\includegraphics[width=.80\textwidth]{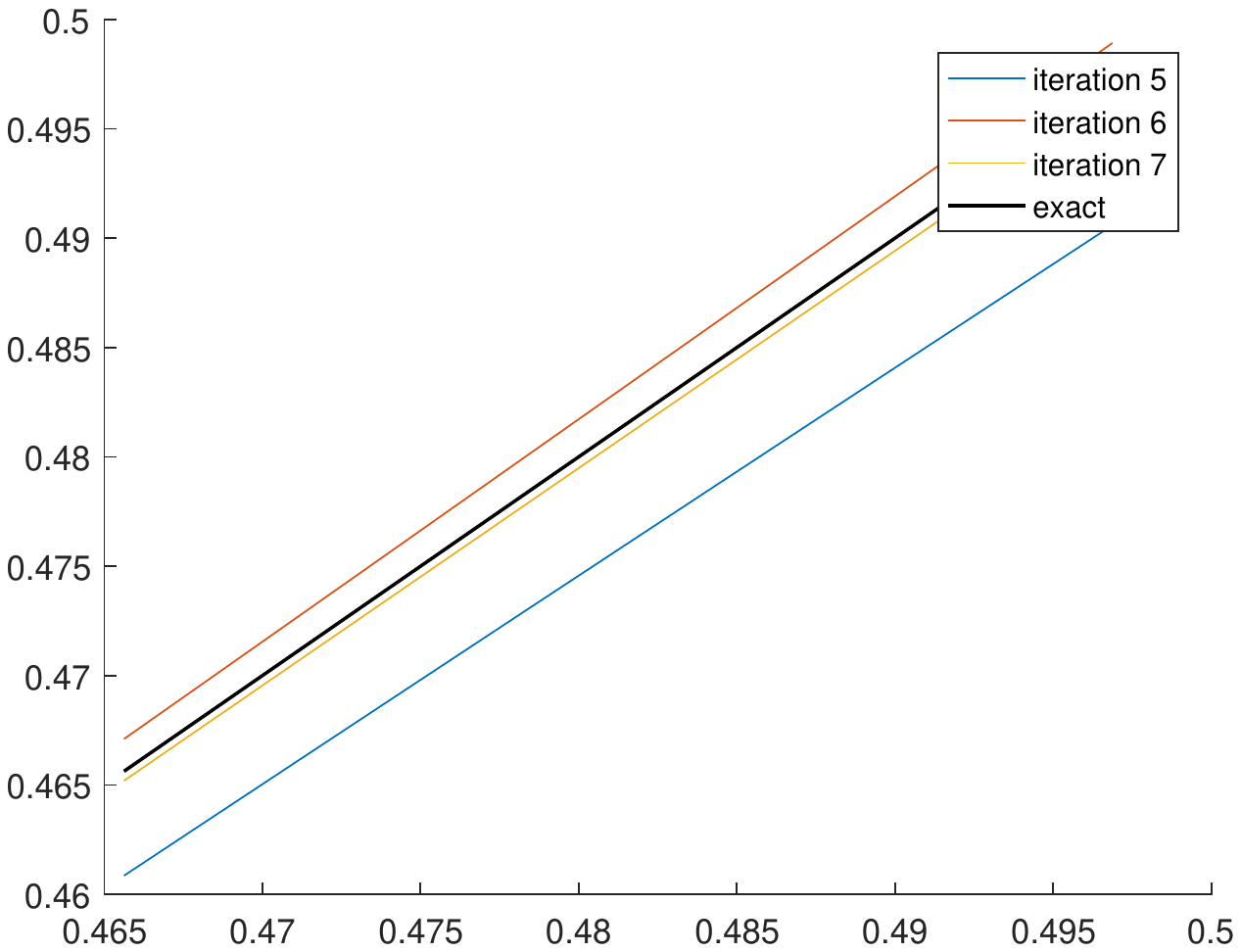}
\caption{Detail from three successive boundaries for example 2. The boundaries oscillate around and converge to the exact.}
\label{fig:oscillting boundariesfig:example2}
\end{figure}

%%%%%%%%%%%%%%%%%%%%%%%%%%%%%%%%%%%%%%%%%%
\section{Alternative Iterative Methods}\label{AltIterative}
Given a boundary for the Stefan problem and the corresponding solution $s,U^s$ one might try to improve on the boundary by considering a direct modification that reduces the discrepancy in the Stefan condition $U^s_x(s(t)t)=-\dot{s}(t)$.  A possible measure the discrepancy is a quadratic one \begin{equation}\label{DiscrMeas1}
D_1(s)=\int_0^T(U^s_x(s(t),t)+\dot{s}(t))^2dt
\end{equation}  One can then try to find an alternative boundary $\hat{s}$ on which this discrepancy is reduced, namely set $\hat{s}=y$ where $y$ is the solution of the following calculus of variations problem in $y$ 
\begin{equation}\label{eq:calcvar1}
\inf_y\int_{0}^{T} (U^s_x(y,t)+\dot{y}(t))^2dt\quad y(0)=0\;y(T)\; free
\end{equation}
which however can be solved by inspection.  

An alternative formulation is to use the discrepancy measure
\begin{equation}\label{DiscrMeas2}
D_2(s)=\int_0^T[(U^s_x(s(t),t)+\dot{s}(t))^2+U^s(s(t),t)^2]dt
\end{equation} 
leading to the calculus of variations problem
\begin{equation}\label{eq:calcvar2}
\inf_y\int_{0}^{T}[(U^s_x(y,t)+y'(t))^2+U^s(y,t)^2]dt\quad y(0)=0\quad y(T)\; free,
\end{equation}
In this formulation we are certain that the updated boundary $\hat{s}=y$ has an overall penalty inferior to $\int_{0}^{T}(U^s_x(s(t),t)+\dot{s}(t))^2dt$ because $\int_{0}^{T}(U^s_x(s(t),t)+\dot{s}(t))^2dt=\int_{0}^{T}[(U^s_x(s(t),t)+\dot{s}(t))^2+U^s(s(t),t)^2]dt$ since $U^s(s(t),t)=0$.  

The Euler Lagrange conditions for the first problem \eqref{eq:calcvar1}  is 
\begin{equation}\label{eq:EulerLagr1}
\frac{d^2 y}{dt^2}=U^s_{xx}(y,t)U^s_{x}(y,t)-U^s_{x,t}(y,t)\qquad y(0)=0,y'(T)=U^s_{x}(y(T),T)
\end{equation}
while for \eqref{eq:calcvar2} is the slightly more complicated:
\begin{equation}\label{eq:EulerLagr2}
\frac{d^2 y}{dt^2}=U^s_{xx}(y,t)U^s_{x}(y,t)-U^s_{x,t}(y,t)+U^s_{x}(y,t)U(y,t)\qquad y(0)=0,y'(T)=U^s_{x}(y(T),T)
\end{equation}
It can be easily shown that \eqref{eq:EulerLagr1} leads to the obvious condition
\begin{equation}\label{eq:EulerLagr1a}
\frac{d y}{dt}=-U^s_{x}(y,t)\qquad y(0)=0,y(T)\;free.
\end{equation}
which requires to find a path $y$ on which $U^s$ satisfies the Stefan condition; however on this path $U^s$ does not necessarily vanish, but we hopefully get a boundary closer to the Stefan solution.  The path satisfying \eqref{eq:EulerLagr2} presents a compromise between the goals of the derivative condition and zero boundary temperature.

Solving even \eqref{eq:EulerLagr1} is complicated by the fact that it requires the values of $U^s$ for $x\ge s(t)$ which is an numerically unstable problem. A similar situation was present in the authors' previous work in the context of locating the exercise boundary of an American typo option by improving on a current exercise boundary \cite{AmPut}.  We developed there an alternative approach sidestepping the need to solve outside the $Q_{s,T}$. Solving \eqref{eq:EulerLagr2} in addition introduces a two point boundary value problem on the numerically generated function $U^s$  that must also be known outside $Q_{s,T}$. We have not worked on either the convergence or the numerical properties of those improvements, but they are in the spirit of the usually efficient (super-linear) policy iteration algorithms analyzed in \cite{AmPut} and it would not be a surprise if they prove more efficient than the $\mathcal{P}$ operator ones that are fixed point algorithms with geometric convergence.  

As in our previous work, we consider a \textit{linearized improvement} that does not require $U^s$ except on the boundary $s$ and could be of use in an alternative method. For a fixed boundary  problem with a solution $U^s$ on  $s$ consider a modification $\hat{s}=s+\xi$ with $\xi$ small.  The first order approximations of interest are $U^s_x(s+\xi,t)\approxeq U^s_x(s,t)+U^s_{xx}(s,t)\xi,\quad U^s(s+\xi,t) \approxeq U^s(s,t)+U^s_x(s,t)\xi=U^s_x(s,t)\xi$. Then we  can write the first order analog of \eqref{eq:calcvar1} as 
\begin{equation}\label{eq:calcvarappr1}
\inf_\xi\int_{0}^{T} (\dot{s}+\dot{\xi}+U^s_x(s,t)+U^s_{xx}(s,t)\xi)^2dt\quad \xi(0)=0\;\xi (T)\;free
\end{equation}
The solution $\xi$ satisfies the obvious differential equation (equivalent to the Euler Lagrange conditions)
\begin{equation}\label{eq:EulerLagrAppr1}
\xi''+\beta(t)\xi+\alpha(t)=0\quad\xi(0),\xi(T)\;free
\end{equation}

The coefficients are given in terms of $U^s(s(t),t)$ and its derivatives, in particular $\beta(t)=U^s_{xx}(s(t),t),\;\alpha(t)=\dot{s}+U^s_{x}(s(t),t)$.  The boundary update could be of the form $\hat{s}=s+\epsilon\xi$ with $\epsilon$ in [0,1] and we can show that for small enough $\epsilon$ we indeed get an improvement.  Consider the discrepancy measure $D_2$ in \eqref{DiscrMeas2} for $s+\epsilon\xi$ $$\int_0^T[(U^s_x(s+\epsilon\xi,t)+\dot{s}+\epsilon\dot{\xi})^2+U^s(s+\epsilon\xi,t)^2]dt$$
Using \eqref{eq:EulerLagrAppr1} we have $$\dot{s}+\epsilon\dot{\xi}+U^s_x(s+\epsilon\xi,t)=\dot{s}+\epsilon(-\dot{s}-U^s_x(s,t)-U^s_{xx}(s,t)\xi)+U^s_x(s,t)+U^s_{xx}(s,t)\epsilon\xi+o(\epsilon^2)=(1-\epsilon)(\dot{s}+U^s_x(s,t))+o(\epsilon^2)$$ and also $$U^s(s+\epsilon\xi,t)=U^s(s,t)+U^s_x(s,t)\epsilon\xi+o(\epsilon^2)=U^s_x(s,t)\epsilon\xi+o(\epsilon^2).$$Hence the quadratic discrepancy for $s+\epsilon\xi$ is \begin{equation}
(1-\epsilon)^2\int_{0}^{T}[ (\dot{s}+U^s_x(s,t))^2+\epsilon^2\xi^2U^s_x(s,t)+o(\epsilon^4)]dt=(1-2\epsilon)\int_{0}^{T} (\dot{s}+U^s_x(s,t))^2dt+o(\epsilon^2)
\end{equation}
This shows that the discrepancy measure $D_2$ is reduced multiplicatively, and thus this boundary updating algorithm will succeed if properly applied.  A similar treatment can be carried out for \eqref{eq:calcvar2}.  
\section{Conclusions}\label{Conclusions}
Concerning the  numerical results presented in \ref{Results} we observe a consistent high speed convergence, not entirely justified by our rather conservative results.  It would be of some interest to extend our theoretical results to explain this fortuitous situation. 

%CONCLUSIONS ..TO BE COMPLETED 
\section*{References}
\bibliographystyle{plain}
\bibliography{mybibStefan}
\end{document}